\documentclass{article}%
\usepackage{amsmath}
\usepackage{amsfonts}
\usepackage{amssymb}
\usepackage{graphicx}
\usepackage{amssymb}
\usepackage{lineno}%
\providecommand{\U}[1]{\protect\rule{.1in}{.1in}}
\newtheorem{lemma}{Lemma}
\newtheorem{theorem}{Theorem}
\newtheorem{conclusion}{Conclusion}

\newtheorem{example}{Example}

\newtheorem{remark}{Remark}
\newenvironment{proof}[1][Proof]{\noindent\textbf{#1.} }{\ \rule{0.5em}{0.5em}}
\begin{document}

\title{The generalized Hadamard product of polynomials and its stability}
\author{Stanis\l aw Bia\l as\\The School of Banking and Management,\\ ul. Armii Krajowej 4, 30-150 Krak\'{o}w, Poland\\
Michal G\'{o}ra\\AGH University of Science and Technology,\\ Faculty of Applied Mathematics,\\ al.
Mickiewicza 30, 30-059 Krak\'{o}w, Poland}
\maketitle

\begin{abstract}
For two polynomials of degrees $n$ and $m$ ($n\geq m$)
\begin{align*}
f\left(  s\right)    & =a_{0}+a_{1}s+\ldots+a_{n-1}s^{n-1}+a_{n}s^{n}\\
g\left(  s\right)    & =b_{0}+b_{1}s+\ldots+b_{m-1}s^{m-1}+b_{m}s^{m}%
\end{align*}
we define a set of polynomials $f\bullet g  =\left\{  F_{0},\ldots,F_{n-m}\right\}  $, where
\[
F_{j}\left(  s\right)  =a_{j}b_{0}+a_{j+1}b_{1}s+\ldots+a_{j+m}b_{m}s^{m},
\]
for $j=0,\ldots,n-m$, and call it \textit{a generalized Hadamard product of $f$ and $g$}.  We give sufficient conditions for the Hurwitz stability of $f\bullet g$. The obtained results show that the famous Garloff--Wagner theorem on the Hurwitz stability of the Hadamard product of polynomials is a special case of a more general fact. We also show that for every polynomial with positive coefficients (even not necessarily stable) one can find a polynomial such that their generalized
Hadamard product is stable. Some connections with polynomials admitting the Hadamard factorization are also given. Numerical examples complete and illustrate the considerations.
\end{abstract}

%\linenumbers

\section{Introduction}

A polynomial is said to be (Hurwitz) stable if all its zeros lie in the open
left-half of the complex plane. In the entire class of polynomials, stable
polynomials occupy a special place. They occur very frequently in many
applications (e.g. in the control theory or in the theory of dynamical
systems) and thus they are important both in theory and in practice. From
among many interesting properties of stable polynomials we recall one result
which is closely related to this work. In 1996 Garloff and Wagner proved in
\cite{Gar-Wag} that the set of Hurwitz stable polynomials is closed under the
Hadamard product. In this paper we show among others that this result is a
special case of a much more general fact.

The work is organized as follows. After preliminary Section 2, we introduce in
Section 3 a generalized Hadamard product of a pair of polynomials. In contrary
to the (usual) Hadamard product, the generalized Hadamard product is not a
polynomial but it is a set of polynomials including, as one of elements, the
(usual) Hadamard product. We give sufficient conditions for its stability. The
results obtained are based on the Garloff--Wagner theorem \cite{Gar-Wag} and
on sufficient conditions for the stability of a real polynomial given by
Katkova and Vishnyakova \cite{Kat-Vis} and Kleptsyn \cite{Kle}. We also show
that for every polynomial  with positive coefficients (but not necessarily stable) one can find a polynomial
such that their generalized Hadamard product is stable. Finally, in Section~4
we give examples completing our considerations and illustrating the results.

\section{Definitions and preliminary results}

In this section, we introduce the basic definitions and notation and remind
some results that will be used throughout this article at various places.

\subsection{Basic notations}

\bigskip We use standard notation: $\mathbb{R}$ and $\mathbb{R}^{n\times n}$
stand for the set of real numbers and for the set of real matrices of order
$n\times n$, respectively; $\mathbb{N}$ denotes the set of positive integers;
$\mathfrak{Re}\left(  \cdot\right)  $ stands for the real part of a complex
number. The degree of a polynomial will be denoted by $\deg\left(
\cdot\right)  $.

\subsection{Stable polynomials and related polynomial families}

\bigskip A polynomial $f$ of degree $n$ $\left(  n\geq1\right)  $%
\begin{equation}
f\left(  s\right)  =a_{0}+a_{1}s+\ldots+a_{n-1}s^{n-1}+a_{n}s^{n} \label{w1}%
\end{equation}
is said to be \textit{Hurwitz stable }(or shortly \textit{stable}) if all its
zeros have negative real parts, and it is said to be \textit{quasi--stable} if
these zeros have non-positive real parts. Together with polynomial (\ref{w1}),
we will consider a polynomial $f^{\ast}$ of the form%
\[
f^{\ast}\left(  s\right)  =a_{n}+a_{n-1}s+\ldots+a_{1}s^{n-1}+a_{0}%
s^{n}\text{.}%
\]
For every nonzero $s$ we have $f\left(  s\right)  =s^{n}f^{\ast}(s^{-1})$ and
thus the polynomial $f$ is stable if and only if $f^{\ast}$ is.

It is well known (and easy verified) that a necessary condition for the
stability of a real polynomial is that its coefficients are all of the same
sign. Without losing generality we will assume in the sequel that they are positive.

Let $\triangle_{i}\left(  f\right)  $ denote the $i$--th leading principal
minor of the Hurwitz matrix $H_{f}\in\mathbb{R}^{n\times n}$ associated with
polynomial (\ref{w1}),%
\[
H_{f}=%
\begin{pmatrix}
a_{n-1} & a_{n} & 0 & 0 & \ldots & 0\\
a_{n-3} & a_{n-2} & a_{n-1} & a_{n} & \ldots & 0\\
a_{n-5} & a_{n-4} & a_{n-3} & a_{n-2} & \ldots & 0\\
\vdots & \vdots & a_{n-5} & a_{n-4} & \ldots & 0\\
\vdots & \vdots & \vdots & \vdots & \ddots & \vdots\\
0 & 0 & 0 & 0 & \ldots & a_{0}%
\end{pmatrix}
\text{,}%
\]
in particular $\triangle_{1}\left(  f\right)  =a_{n-1}$ and $\triangle
_{n}\left(  f\right)  =\det H_{f}=a_{0}\triangle_{n-1}\left(  f\right)  $. It
follows from the Routh--Hurwitz criterion (see, for example, Gantmacher
\cite{Gan}) that polynomial (\ref{w1}) with positive coefficients is stable if
and only if $\triangle_{i}\left(  f\right)  >0,$ for $i=1,2,\ldots,n-1$.

For the simplicity of notation, we introduce the following sets of polynomials:

\begin{itemize}
\item $\mathbb{R}_{n}^{+}=\left\{  s\rightarrow a_{n}s^{n}+\ldots+a_{1}%
s+a_{0}:a_{i}>0\ \left(  i=0,\ldots,n\right)  \right\}  ,$

\item $\mathcal{H}_{n}=\left\{  f\in\mathbb{R}_{n}^{+}:f\left(  s\right)
=0\Rightarrow\mathfrak{Re}\left(  s\right)  <0\right\}  ,$

\item $\overline{\mathcal{H}_{n}}=\left\{  f\in\mathbb{R}_{n}^{+}:f\left(
s\right)  =0\Rightarrow\mathfrak{Re}\left(  s\right)  \leq0\right\}  $
\end{itemize}

and for $n\geq3$ and $\alpha\in\mathbb{R}$

\begin{itemize}
\item $\mathcal{W}_{n}=\{f\in\mathbb{R}_{n}^{+}:\lambda_{i}\left(  f\right)
<1$ $\left(  i=2,\ldots,n-1\right)  \},$

\item $\mathcal{W}_{n}^{\alpha}=\{f\in\mathbb{R}_{n}^{+}:\lambda_{i}\left(
f\right)  <\alpha$ $\left(  i=2,\ldots,n-1\right)  \},$

\item $\mathcal{V}_{n}=\{f\in\mathbb{R}_{n}^{+}:\lambda_{2}\left(  f\right)
+\ldots+\lambda_{n-1}\left(  f\right)  <1\},$
\end{itemize}

where, for $f$ of the form (\ref{w1}), the positive numbers $\lambda
_{2}\left(  f\right)  ,\ldots,\lambda_{n-1}\left(  f\right)  $ are defined as
\begin{equation}
\lambda_{i}\left(  f\right)  =\frac{a_{i-2}a_{i+1}}{a_{i}a_{i-1}}\text{,\quad
for }i=2,\ldots,n-1\text{.} \label{w101}%
\end{equation}
Connections between sets $\mathcal{W}_{n}$, $\mathcal{W}_{n}^{\alpha}$,
$\mathcal{V}_{n}$ and $\mathcal{H}_{n}$ will be discussed in detail in Section~\ref{Sec3}.

\subsection{The Hadamard product of polynomials}

\bigskip Let $f$ be as in (\ref{w1}) and let $g\in\mathbb{R}_{m}^{+}$ be of
the form%
\begin{equation}
g\left(  s\right)  =b_{0}+b_{1}s+\ldots+b_{m-1}s^{m-1}+b_{m}s^{m}\text{.}
\label{w2}%
\end{equation}
Supposing that $m=\deg\left(  g\right)  \leq\deg\left(  f\right)  =n$ we can
define the polynomial $f\circ g\in\mathbb{R}_{m}^{+}$ of the form%
\[
\left(  f\circ g\right)  \left(  s\right)  =a_{0}b_{0}+a_{1}b_{1}%
s+\ldots+a_{m-1}b_{m-1}s^{m-1}+a_{m}b_{m}s^{m}%
\]
called \textit{the Hadamard product of polynomials} $f$ \textit{and} $g$.

The Hadamard product of polynomials has been studied by many authors, but we
will mention here only one work that we will refer to many times in the
sequel. Namely, Garloff~and~Wagner~considered in \cite{Gar-Wag} the stability
problem for the Hadamard product of two real polynomials and obtained, among
others, the following

\begin{theorem}
[Garloff, Wagner]Let $m\leq n$ be positive integers. If $f\in\mathcal{H}_{n}$
and $g\in\mathcal{H}_{m}$,$\ $then $f\circ g\in\mathcal{H}_{m}$.
\end{theorem}

\section{The generalized Hadamard product of polynomials and its stability}

\label{Sec3} \bigskip Here and below in this section polynomials $f$ and $g$
are of the form (\ref{w1}) and (\ref{w2}), respectively, and such that
$m=\deg\left(  g\right)  \leq\deg\left(  f\right)  =n$. Let $f_{0}%
,\ldots,f_{n-m}\in\mathbb{R}_{m}^{+}$ be of the form%
\begin{equation}
f_{j}\left(  s\right)  =a_{j}+a_{j+1}s+\ldots+a_{j+m}s^{m},\text{ for
}j=0,\ldots,n-m\text{.} \label{w22}%
\end{equation}
Then, \textit{the generalized Hadamard product of a pair of polynomials $f$
and $g$} is defined as a set of polynomials
\[
f\bullet g=\left\{  F_{0},\ldots,F_{n-m}\right\}  ,
\]
where $F_{j}=f_{j}\circ g$ ($j=0,\ldots,n-m$).

Note that the generalized Hadamard product is a generalization of the (usual)
Hadamard product. Indeed, if $\deg f=\deg g$ then $f\bullet g=\left\{  f\circ
g\right\}  ,$ and in general case (i.e. if $\deg f\leq\deg g$) we have $f\circ
g=F_{0}\in f\bullet g$.

In our main results, we want to focus the attention on the stability of the
generalized Hadamard product of polynomials. As Example~3 shows, the stability
of the polynomial $f$ does not imply those of $f_{0},\ldots,f_{n-m}$ and, in a
consequence, the stability of the generalized Hadamard product of two
\textit{stable} polynomials is not an immediate consequence of the
Garloff--Wagner theorem.

\subsection{Main results}

\label{Subsec3} \bigskip As regards $\mathcal{H}_{n}$, the set of stable
polynomials of degree $n$, it is easy to see that $\mathcal{H}_{1}%
=\mathbb{R}_{1}^{+}$ and $\mathcal{H}_{2}=\mathbb{R}_{2}^{+}$. Moreover,
Kemperman proved in \cite{Kem} that each principal submatrix (i.e. an $\left(
n-k\right)  $-by-$\left(  n-k\right)  $ matrix obtained from a given
$n$-by-$n$ matrix by removing its $k$ rows and the same $k$ columns) of the
Hurwitz matrix associated with a stable polynomial has a positive determinant.
It allows us to conclude that for $n\geq3$%
\begin{equation}
\mathcal{H}_{n}\subset\mathcal{W}_{n}\text{.} \label{w11}%
\end{equation}
For $n=3$, by the Routh--Hurwitz criterion, we have $\mathcal{H}%
_{3}=\mathcal{W}_{3}$, and in general case inclusion (\ref{w11}) is proper
(see Section~\ref{Sec_ex} for numerical examples completing the results). The
set $\mathcal{W}_{n}$ has been defined only for $n\geq3$, but to simplify the
formulation of our main results we put by definition $\mathcal{W}%
_{k}=\mathcal{H}_{k}$ for $k=1,2$.

We start with the case in which the degree of the generalized Hadamard product
(understood in a natural way as a common degree of all its elements) does not
exceed 4. Note that the polynomial $f$ that occurs in Theorem~\ref{th1} does
not need to be stable.

\begin{theorem}
\label{th1}Let $m,n\in\mathbb{N}$ be such that $1\leq m\leq4$ and $m\leq n$.
If $f\in\mathcal{W}_{n}\ $and $g\in\mathcal{H}_{m}$, then $f\bullet
g\subset\mathcal{H}_{m}$.
\end{theorem}

\begin{proof}
It follows from the assumption, that the polynomials

\begin{itemize}
\item $A_{j}\left(  s\right)  =a_{j}+a_{j+1}s,$ for $j=0,\ldots,n-1$;

\item $B_{j}\left(  s\right)  =A_{j}\left(  s\right)  +a_{j+2}s^{2},$ for
$j=0,\ldots,n-2$ and $n\geq2$;

\item $C_{j}\left(  s\right)  =B_{j}\left(  s\right)  +a_{j+3}s^{3},$ for
$j=0,\ldots,n-3$ and $n\geq3$
\end{itemize}

are stable. Thus, for $1\leq m\leq3$ the result follows from the
Garloff--Wagner theorem. For $m=4$, we have $f\bullet g=\left\{  F_{0}%
,\ldots,F_{n-4}\right\}  $ where%
\begin{equation}
F_{j}\left(  s\right)  =a_{j}b_{0}+a_{j+1}b_{1}s+a_{j+2}b_{2}s^{2}%
+a_{j+3}b_{3}s^{3}+a_{j+4}b_{4}s^{4}\text{,} \label{w23}%
\end{equation}
for $j=0,\ldots,n-4$. The Hurwitz matrix associated with the polynomial
$F_{j}$ has the form%
\[
H_{F_{j}}=\left(
\begin{array}
[c]{cccc}%
a_{j+3}b_{3} & a_{j+4}b_{4} & 0 & 0\\
a_{j+1}b_{1} & a_{j+2}b_{2} & a_{j+3}b_{3} & a_{j+4}b_{4}\\
0 & a_{j}b_{0} & a_{j+1}b_{1} & a_{j+2}b_{2}\\
0 & 0 & 0 & a_{j}b_{0}%
\end{array}
\right)  \text{.}%
\]
We will show that all its leading principal minors are positive.

Obviously, $\triangle_{1}\left(  F_{j}\right)  =a_{j+3}b_{3}$. Moreover, since
$f\in\mathcal{W}_{n}$, we have
\begin{align*}
\triangle_{2}\left(  F_{j}\right)   &  =a_{j+2}a_{j+3}b_{2}b_{3}%
-a_{j+1}a_{j+4}b_{1}b_{4}>a_{j+2}a_{j+3}\left(  b_{2}b_{3}-b_{1}b_{4}\right)
=\\
&  =a_{j+2}a_{j+3}\triangle_{2}\left(  g\right)
\end{align*}
and%
\begin{align*}
\triangle_{3}\left(  F_{j}\right)   &  =a_{j+1}a_{j+2}a_{j+3}b_{1}b_{2}%
b_{3}-a_{j}a_{j+3}^{2}b_{0}b_{3}^{2}-a_{j+1}^{2}a_{j+4}b_{1}^{2}b_{4}>\\
&  >a_{j+1}a_{j+2}a_{j+3}\left(  b_{1}b_{2}b_{3}-b_{0}b_{3}^{2}-b_{1}^{2}%
b_{4}\right)  =a_{j+1}a_{j+2}a_{j+3}\triangle_{3}\left(  g\right)  .
\end{align*}
Since $g$ is stable, the result follows from the Routh--Hurwitz criterion.
\end{proof}

In the general case Theorem~\ref{th1} does not hold (see Example~\ref{ex4})
but the following is true:

\begin{theorem}
\label{th2}Let $m\leq n$ be positive integers. If $f\in\mathcal{H}_{n}\ $and
$g\in\mathcal{H}_{m}$, then

\begin{description}
\item[(a)] $F_{0},F_{n-m}\in\mathcal{H}_{m};$

\item[(b)] $f\bullet g\subset\overline{\mathcal{H}_{m}}$.
\end{description}
\end{theorem}

To prove this theorem we need two auxiliary lemmas.

\begin{lemma}
\label{lem1}Let $N>n$ be positive integers and let $f\in\mathcal{H}_{n}$. Then
for every $\varepsilon>0$ there exist positive numbers $a_{n+1},\ldots,a_{N}$
such that:

\begin{description}
\item[(a)] $a_{n+1},\ldots,a_{N}\in\left(  0,\varepsilon\right)  ;$

\item[(b)] a polynomial
\[
F_{\varepsilon}\left(  s\right)  =f\left(  s\right)  +s^{n+1}\left(
a_{n+1}+a_{n+2}s+\ldots+a_{N}s^{N-n-1}\right)
\]
is stable.
\end{description}
\end{lemma}

\begin{proof}
Let us fix any $\varepsilon>0$. We will proceed by induction with respect to
$N-n$. For $N-n=1$ the result follows from Lemma 5.3 in Bhattacharyya
\textit{et al.} \cite{Bhat} (for the sake of completeness of this work, we
present in Appendix a proof of this fact based on the Routh--Hurwitz
criterion). Hence, let us assume that the result holds for $N-n=k$ and we need
to prove it for $N-n=k+1$.

By the induction assumption, there exist positive numbers $a_{n+1}%
,\ldots,a_{n+k}\in\left(  0,\varepsilon\right)  $ such that the polynomial%
\[
F_{\varepsilon,k}\left(  s\right)  =f\left(  s\right)  +a_{n+1}s^{n+1}%
+\ldots+a_{n+k}s^{n+k}%
\]
is stable. The argument used in the first induction step allows us to conclude
again that for the polynomial $F_{\varepsilon,k}$ we can find $a_{n+k+1}%
\in\left(  0,\varepsilon\right)  $ such that the polynomial%
\[
F_{\varepsilon,k+1}\left(  s\right)  =F_{\varepsilon,k}\left(  s\right)
+a_{n+k+1}s^{n+k+1}%
\]
is also stable, proving that the thesis is true for $N-n=k+1$. It means that
it is true for every $N>n$.
\end{proof}

\begin{lemma}
\label{lem2}Let $f\in\mathcal{H}_{n}$ and $k\in\mathbb{N}$. There exists a
sequence of polynomials $\left\{  p_{m}\right\}  _{m\in\mathbb{N}}%
\subset\mathbb{R}_{k-1}^{+}$ satisfying the following conditions:

\begin{description}
\item[(a)] for every $m\in\mathbb{N}$, the polynomial
\[
s\rightarrow p_{m}\left(  s\right)  +s^{k}f\left(  s\right)
\]
is stable;

\item[(b)] $\lim\limits_{m\rightarrow\infty}p_{m}=0$.
\end{description}
\end{lemma}

\begin{proof}
We have noticed before, that if the polynomial $f$ is stable, then so is
$f^{\ast}$. When applying Lemma \ref{lem1} to $f^{\ast}$ we get that for every
$N>n$ and for every $\varepsilon>0$ there exist positive numbers
$a_{n+1},\ldots,a_{N}\in\left(  0,\varepsilon\right)  $ such that the
polynomial%
\[
F_{\varepsilon}\left(  s\right)  =f^{\ast}\left(  s\right)  +s^{n+1}\left(
a_{n+1}+\ldots+a_{N}s^{N-n-1}\right)
\]
is stable. The stability of $F_{\varepsilon}$ is equivalent to the stability
of $F_{\varepsilon}^{\ast}$, i.e.
\[
F_{\varepsilon}^{\ast}\left(  s\right)  =a_{N}+a_{N-1}s+\ldots+a_{n+1}%
s^{N-n-1}+s^{N-n}f\left(  s\right)  .
\]
Letting $k=N-n$, $\varepsilon=\frac{1}{m}$ (for $m\in\mathbb{N}$) and
\[
p_{m}\left(  s\right)  =b_{0}^{\left(  m\right)  }+b_{1}^{\left(  m\right)
}s+\ldots+b_{k-1}^{\left(  m\right)  }s^{k-1},
\]
where $b_{i}^{\left(  m\right)  }=a_{N-i}$ (for $i=0,\ldots,N-n-1$), the
result follows from Lemma~\ref{lem1} by applying it to every $m\in\mathbb{N}$.
\end{proof}

From Lemma~\ref{lem2} one can easily draw the following conclusion.

\begin{conclusion}
\label{con1}If $f\in\mathcal{H}_{n}$, then for every $k\in\mathbb{N}$ there
exists a polynomial $p\in\mathbb{R}_{k-1}^{+}$ such that the polynomial
$s\rightarrow p\left(  s\right)  +s^{k}f\left(  s\right)  $ is stable.
\end{conclusion}

\noindent\textbf{Proof of Theorem \ref{th2}} In order to prove the first part
of the theorem, note that the stability of $F_{0}$ follows from
the~Garloff--Wagner~theorem. Similarly, the stability of $F_{n-m}$ is a
consequence of the identity%
\[
F_{n-m}=\left(  f^{\ast}\circ g^{\ast}\right)  ^{\ast}\text{,}%
\]
and again of the Garloff--Wagner theorem.

To prove the second part, let $f\bullet g=\left\{  F_{0},\ldots,F_{n-m}%
\right\}  ,$ where%
\[
F_{j}\left(  s\right)  =a_{j}b_{0}+a_{j+1}b_{1}s+\ldots+a_{j+m}b_{m}%
s^{m}\text{,}%
\]
for $j=0,\ldots,n-m$. Since we have just shown the stability of $F_{0}$ and
$F_{n-m},$ let us fix any $j\in\left\{  1,\ldots,n-m-1\right\}  $. The
stability of $g$ and Lemma~\ref{lem2} imply that there exists a sequence of
polynomials $\left\{  p_{m}\right\}  _{m\in\mathbb{N}}\subset\mathbb{R}%
_{j-1}^{+}$ such that the polynomials%
\[
G_{m}\left(  s\right)  =p_{m}\left(  s\right)  +s^{j}g\left(  s\right)
\text{,\quad}m\in\mathbb{N}%
\]
are stable. By the Garloff--Wagner theorem, the stability of $G_{m}$ implies
the stability of $G_{m}\circ f$. Since $p_{m}\rightarrow0$ (as $m\rightarrow
\infty)$, we obtain that $G_{m}\circ f\rightarrow G$ (as $m\rightarrow\infty
$), where $G\left(  s\right)  =s^{j}F_{j}\left(  s\right)  $. By the
continuous dependence of zeros of a polynomial on its coefficients, we obtain
the quasi--stability of the polynomial $F_{j}$. This completes the proof.

\subsection{Further extensions}

\bigskip In this subsection we shall give sufficient conditions for the
stability of the generalized Hadamard product $f\bullet g$, supposing that
polynomials $f$ and $g$ satisfy some additional, more or less restrictive, conditions.

\subsubsection{Polynomials admitting a Hadamard factorization}

\bigskip Recall that the polynomial $f\in\mathcal{H}_{n}$ admits a Hadamard
factorization if there exist two polynomials $f_{1},f_{2}\in\mathcal{H}_{n}$
for which $f=f_{1}\circ f_{2}$. It is easy to see that all stable polynomials
of degree $2$ and $3$ have Hadamard factorizations (see Garloff and
Shrinivasan \cite{Gar-Shr}), but it is also known that there exist polynomials
of degree $4$ that do not have a Hadamard factorization (again Garloff and
Shrinivasan \cite{Gar-Shr}). Giving some characterization of polynomials
admitting a Hadamard factorization in general case is, to the best of our
knowledge, an open problem, but some necessary conditions for the Hadamard
factorization of stable polynomials can be found in Loredo--Villalobos and
Aguirre--Hern\'{a}ndez \cite{LoV} (see also Remark~\ref{rem2} below for some
sufficient condition for the Hadamard factorization of a polynomial).

\begin{theorem}
\label{th3}Let $m\leq n$ be positive integers. If $f\in\mathcal{H}_{n}$ and
$g$ has a Hadamard factorization, i.e. $g=g_{1}\circ g_{2}$ for some
$g_{1},g_{2}\in\mathcal{H}_{m}$, then $f\bullet g\subset\mathcal{H}_{m}$.
\end{theorem}

\begin{proof}
Recall that $f\bullet g=\left\{  F_{0},\ldots,F_{n-m}\right\}  $. The
stability of $F_{0}$ and $F_{n-m}$ follows from Theorem \ref{th2}, and thus
let $j\in\left\{  1,\ldots,n-m-1\right\}  $. It follows from Conclusion
\ref{con1} that there exists a polynomial $p\in\mathbb{R}_{j-1}^{+}$, say
$p\left(  s\right)  =p_{0}+p_{1}s+\ldots+p_{j-1}s^{j-1}$, such that the
polynomial $G_{j}\left(  s\right)  =p\left(  s\right)  +s^{j}g_{1}\left(
s\right)  $ is stable. By the Garloff--Wagner theorem we get the stability of
$\left(  \left(  G_{j}\circ f\right)  ^{\ast}\circ g_{2}^{\ast}\right)
^{\ast}$. It suffices to note that $\left(  \left(  G_{j}\circ f\right)
^{\ast}\circ g_{2}^{\ast}\right)  ^{\ast}=F_{j}$.

Indeed, the polynomials $g_{1}$ and $g_{2}$ can be written as
\[
g_{1}\left(  s\right)  =\beta_{0}+\beta_{1}s+\ldots+\beta_{m-1}s^{m-1}%
+\beta_{m}s^{m}%
\]
and%
\[
g_{2}\left(  s\right)  =\frac{b_{0}}{\beta_{0}}+\frac{b_{1}}{\beta_{1}%
}s+\ldots+\frac{b_{m-1}}{\beta_{m-1}}s^{m-1}+\frac{b_{m}}{\beta_{m}}s^{m}%
\]
and thus%
\begin{align*}
\left(  G_{j}\circ f\right)  \left(  s\right)   &  =a_{0}p_{0}+\ldots
+a_{j-1}p_{j-1}s^{j-1}+\\
&  +s^{j}\left(  a_{j}\beta_{0}+a_{j+1}\beta_{1}s+\ldots+a_{j+m}\beta_{m}%
s^{m}\right)  ,
\end{align*}
or equivalently,%
\begin{align*}
\left(  G_{j}\circ f\right)  ^{\ast}\left(  s\right)   &  =a_{j+m}\beta
_{m}+a_{j+m-1}\beta_{m-1}s+\ldots+a_{j}\beta_{0}s^{m}+\\
&  +a_{j-1}p_{j-1}s^{m+1}\ldots+a_{0}p_{0}s^{j+m}\text{.}%
\end{align*}
Since%
\[
\left(  \left(  G_{j}\circ f\right)  ^{\ast}\circ g_{2}^{\ast}\right)  \left(
s\right)  =a_{j+m}b_{m}+a_{j+m-1}b_{m-1}s+\ldots+a_{j}b_{0}s^{m},
\]
then $\left(  \left(  G_{j}\circ f\right)  ^{\ast}\circ g_{2}^{\ast}\right)
^{\ast}\left(  s\right)  =\left(  f_{j}\circ g\right)  \left(  s\right)  $
where $f_{j}$ are as in (\ref{w22}). This completes the proof.
\end{proof}

\subsubsection{Polynomials from $\mathcal{W}_{n}^{\alpha}$}

\bigskip One can show that for every positive number $\alpha$ the set
$\mathcal{W}_{n}^{\alpha}$ is non-empty (see Lemma \ref{lem3} below).
Moreover, Katkova and Vishnyakova proved in \cite{Kat-Vis} that a real
polynomial (\ref{w1}) of degree $4$ or more and satisfying inequalities%
\[
\alpha^{\ast}a_{i}a_{i-1}-a_{i-2}a_{i+1}>0\text{\quad}\left(  i=2,\ldots
,n-1\right)  ,
\]
where $\alpha^{\ast}\approx0.46557$ is the unique real solution to the
equation
\begin{equation}
1=\alpha\left(  1+\alpha\right)  ^{2}, \label{w11a}%
\end{equation}
is stable. This observation together with our earlier considerations in
Subsection~\ref{Subsec3} lead to the inclusion%
\begin{equation}
\mathcal{W}_{n}^{\alpha^{\ast}}\subset\mathcal{H}_{n} \label{w12}%
\end{equation}
being true for $n\geq3$. For simplicity, as in case of $\mathcal{W}_{n}$, we
put by definition $\mathcal{W}_{k}^{\alpha}=\mathcal{H}_{k}$ for every
$\alpha\in\left(  0,1\right)  $ and $k=1,2$.

Let us also note that for $\beta^{\ast}=\sqrt{\alpha^{\ast}}\approx0.68233$ we
have%
\[
\mathcal{W}_{n}^{\alpha^{\ast}}\subset\mathcal{W}_{n}^{\beta^{\ast}}%
\]
but, as follows from Example 4, in contrary to $\mathcal{W}_{n}^{\alpha^{\ast
}}$, $\mathcal{W}_{n}^{\beta^{\ast}}$ contains unstable polynomials.

The following theorem states, among others, a sufficient condition for the
stability of the generalized Hadamard product $f\bullet g$ in the case when
neither $f$ nor $g$ is stable.

\begin{theorem}
\label{th4}Let $m\leq n$ be positive integers, let $\alpha^{\ast}$ be the
unique real solution to the equation (\ref{w11a}) and let $\beta^{\ast}%
=\sqrt{\alpha^{\ast}}$. Then

\begin{description}
\item[(a)] if $f\in\mathcal{W}_{n}$ and $g\in\mathcal{W}_{m}^{\alpha^{\ast}}$,
then $f\bullet g\subset\mathcal{W}_{m}^{\alpha^{\ast}}$;

\item[(b)] if $f\in\mathcal{W}_{n}^{\beta^{\ast}}$ and $g\in\mathcal{W}%
_{m}^{\beta^{\ast}}$, then $f\bullet g\subset\mathcal{W}_{m}^{\alpha^{\ast}}$.
\end{description}
\end{theorem}

\begin{proof}
We have $f\bullet g=\left\{  F_{0},\ldots,F_{n-m}\right\}  $ where%
\[
F_{j}\left(  s\right)  =A_{0,j}+A_{1,j}s+\ldots+A_{m,j}s^{m}%
\]
with $A_{i,j}=a_{j+i}b_{i}$ for $i=0,\ldots,m,$ $j=0,\ldots,n-m$. Then, in
view of (\ref{w101}), we have for $i=2,\ldots,m-1$:
\begin{equation}
\lambda_{i}\left(  F_{j}\right)  =\frac{A_{i-2,j}A_{i+1,j}}{A_{i,j}A_{i-1,j}%
}=\frac{a_{j+i-2}a_{j+i+1}}{a_{j+i}a_{j+i-1}}\frac{b_{i-2}b_{i+1}}%
{b_{i}b_{i-1}}=\lambda_{i+j}\left(  f\right)  \lambda_{i}\left(  g\right)
\label{ww3}%
\end{equation}
what, in both cases (a) and (b), by the assumptions on $f$ and $g$, leads to
an inequality
\[
\lambda_{i}\left(  F_{j}\right)  <\alpha^{\ast}%
\]
that completes the proof.
\end{proof}

\subsubsection{Polynomials from $\mathcal{V}_{n}$}

\bigskip As regards the set $\mathcal{V}_{n}$, it is clear that $\mathcal{V}%
_{3}=\mathcal{W}_{3}=\mathcal{H}_{3}$. As previously, we put by definition
$\mathcal{V}_{k}=\mathcal{H}_{k}$ for $k=1,2$. It can be also shown that
$\mathcal{V}_{4}=\mathcal{H}_{4}$ (see Proposition 7 in Bia\l as and
Bia\l as--Cie\.{z} \cite{Bia-Cie}) and for $n\geq5$ it holds%
\[
\mathcal{V}_{n}\subsetneq\mathcal{W}_{n}\text{.}%
\]
Moreover, for $n=3$ and $n=4$ we have $\mathcal{W}_{n}^{\alpha^{\ast}}%
\subset\mathcal{V}_{n}$, but in the general case there is no inclusion between
$\mathcal{W}_{n}^{\alpha^{\ast}}\ $and $\mathcal{V}_{n}$. Kleptsyn \cite{Kle}
proved, in turn, that for $n\geq3$ a sufficient condition for the stability of
$f$ is%
\[
\lambda_{2}\left(  f\right)  +\ldots+\lambda_{n-1}\left(  f\right)  <1,
\]
where $\lambda_{2}\left(  f\right)  ,\ldots,\lambda_{n-1}\left(  f\right)  $
are given by (\ref{w101}). In other words, for $n\geq3$ we have
\[
\mathcal{V}_{n}\subset\mathcal{H}_{n}.
\]
It allows us to give one more sufficient condition for the stability of the
generalized Hadamard product of two polynomials.

\begin{theorem}
\label{th5}Let $m\leq n$ be positive integers. If $f\in\mathcal{W}_{n}$ and
$g\in\mathcal{V}_{m}$, then $f\bullet g\subset\mathcal{V}_{m}$.
\end{theorem}

\begin{proof}
The proof is similar to that of Theorem \ref{th4} and follows from the easily
verified identities%
\[
\sum_{i=2}^{m-1}\lambda_{i}\left(  F_{j}\right)  =\sum_{i=2}^{m-1}%
\lambda_{i+j}\left(  f\right)  \lambda_{i}\left(  g\right)  <\sum_{i=2}%
^{m-1}\lambda_{i}\left(  g\right)  <1
\]
which hold for $j=0,\ldots,n-m$.
\end{proof}

\subsection{Stabilization by the Hadamard product}

\bigskip We will prove now one more property of the generalized Hadamard
product of polynomials. Namely, we will show that for every polynomial $f$
there exists a stable polynomial $g$ for which the generalized Hadamard
product $f\bullet g$ becomes stable.

\begin{theorem}
\label{th7}Suppose that $f\in\mathbb{R}_{n}^{+}$. Then for every $m\in\left\{
1,\ldots,n\right\}  $ there exists a polynomial $g\in\mathcal{H}_{m}$ such
that $f\bullet g\subset\mathcal{H}_{m}$.
\end{theorem}

The proof is based on the following observation.

\begin{lemma}
\label{lem3}For every integer $m\geq3$ and for every $\varepsilon>0$ there
exists a polynomial $g\in\mathbb{R}_{m}^{+}$ such that $\lambda_{i}\left(
g\right)  =\varepsilon$, for $i=2,\ldots,m-1$.
\end{lemma}

\begin{proof}
It is easy to see that the polynomial $g\left(  s\right)  =b_{0}+b_{1}%
s+\ldots+b_{m}s^{m}$ with coefficients given by the following recurrence
formulae: $b_{0},b_{1},b_{2}$ -- arbitrary positive numbers and
\[
b_{k+2}=\varepsilon\frac{b_{k+1}b_{k}}{b_{k-1}}\text{,\quad for }%
k=1,\ldots,m-2
\]
satisfies our requirement.
\end{proof}

\noindent\textbf{Proof of Theorem \ref{th7} }Suppose that $n\geq3$ (for $1\leq
n\leq2$ the result is not interesting) and fix any $m\in\left\{
3,\ldots,n\right\}  $. Since $f\bullet g=\left\{  F_{0},\ldots,F_{n-m}%
\right\}  $ where, according to (\ref{ww3}),
\[
\lambda_{i}\left(  F_{j}\right)  =\lambda_{i+j}\left(  f\right)  \lambda
_{i}\left(  g\right)  \text{,}%
\]
the result follows from Lemma \ref{lem3}\textbf{ }by applying it to any
$\varepsilon<\frac{\alpha^{\ast}}{\max_{2\leq i\leq n-1}\lambda_{i}\left(
f\right)  }$.

\bigskip It seems to be interesting that if $f\in\mathcal{W}_{n}$ (it is still
not necessarily stable), then the polynomial $g$ that occurs in Theorem
\ref{th7} can be chosen as the one having a Hadamard factorization.

\begin{theorem}
\label{th6}Suppose that $f\in\mathcal{W}_{n}$. Then for every $m\in\left\{
1,\ldots,n\right\}  $ there exists a polynomial $g\in\mathcal{H}_{m}$ having a
Hadamard factorization and such that $f\bullet g\subset\mathcal{H}_{m}$.
\end{theorem}

\begin{proof}
Since for $1\leq n\leq3$ and $1\leq m\leq2$ the result is obvious, suppose
that $n$ $\geq4$ and take any $m\in\left\{  3,\ldots,n\right\}  $. Also, let
$f_{j}$ be given by (\ref{w22}). It follows from the assumption that
$\lambda_{i}\left(  f_{j}\right)  <1$ for $i=2,\ldots,m-1$; $j=0,\ldots,n-m$.
Bia\l as and Bia\l as--Cie\.{z} proved recently (see Theorems 8 and 10 in
\cite{Bia-Cie}) that if $f\in\mathcal{W}_{n}$, then there exists a positive
number $p^{\ast}$ such that for every $p>p^{\ast}$ the $p$--th Hadamard power
of $f$, i.e.
\[
f^{[p]}\left(  s\right)  =a_{0}^{p}+a_{1}^{p}s+\ldots+a_{n-1}^{p}s^{n-1}%
+a_{n}^{p}s^{n}%
\]
is stable. Thus, it follows from the assumption that for every $j=0,\ldots
,n-m$ there exists a positive number $p_{j}^{\ast}$ such that the polynomial
$f_{j}^{[p]}$ is stable for all $p>p_{j}^{\ast}$. Letting $p^{\ast}%
=\max\{p_{0}^{\ast},\ldots,p_{n-m}^{\ast}\}$, we get that all the polynomials
$f_{0}^{[p]},\ldots,f_{n-m}^{[p]}$ are stable for all $p>p^{\ast}$. Fix now
any $p>p^{\ast}$ and define the polynomial $g\in\mathbb{R}_{m}^{+}$ as
follows
\[
g=f_{0}^{[p]}\circ\ldots\circ f_{n-m}^{[p]}\text{.}%
\]
By the Garloff--Wagner theorem, the polynomial $g$ is stable. Moreover, it is
easy to see that the polynomials $g\circ f_{j}$ are also stable for
$j=0,\ldots,n-m$. This completes the proof.
\end{proof}

\begin{remark}
Bia\l as and Bia\l as--Cie\.{z} proved in \cite{Bia-Cie} that for $n\geq3$ and
$f\in\mathcal{W}_{n}$ the positive number $p^{\ast}$ that has occurred in the
proof of Theorem \ref{th6} can be calculated as%
\begin{equation}
p^{\ast}=\frac{\log\alpha^{\ast}}{\log\max_{2\leq i\leq n-1}\lambda_{i}\left(
f\right)  }\text{,} \label{ww1}%
\end{equation}
where $\alpha^{\ast}$ is given by (\ref{w11a}).
\end{remark}

\begin{remark}
\label{rem2} If, for $p^{\ast}$ as in (\ref{ww1}), we have $p^{\ast}<0.5$,
i.e. if
\begin{equation}
\max_{2\leq i\leq n-1}\lambda_{i}\left(  f\right)  <\gamma^{\ast}%
\approx0.216\,76\text{,} \label{ww2}%
\end{equation}
where $\gamma^{\ast}$ is the unique real solution to the equation
$\gamma\left(  \gamma-1\right)  ^{2}=1-4\gamma$, then the polynomial
$f^{[0.5]}\ $is stable. It means that inequality (\ref{ww2}) is a sufficient
condition for the polynomial $f$ to have a Hadamard factorization.
\end{remark}

\section{Examples}

\label{Sec_ex}

In this last part of the paper we show a few examples completing and
illustrating the results presented in the previous sections.

\begin{example}
\label{ex1}Let $f\in\mathbb{R}_{4}^{+}$ be a polynomial of the form%
\[
f\left(  s\right)  =2s^{4}+2s^{3}+4s^{2}+2s+3.
\]
Since%
\[
\lambda_{2}\left(  f\right)  =0.75\text{\quad and\quad}\lambda_{3}\left(
f\right)  =0.5,
\]
it follows that $f\in\mathcal{W}_{4}$. On the other hand, calculations show
that $\triangle_{3}\left(  f\right)  =-4$ and thus $f$ is not stable. It means
that the inclusion $\mathcal{H}_{4}\subset\mathcal{W}_{4}$ is proper. Other
similar examples may be constructed for $n>4$.
\end{example}

\begin{example}
\label{ex2}Let $f\in\mathbb{R}_{3}^{+}$ be a polynomial of the form%
\[
f\left(  s\right)  =s^{3}+3s^{2}+7s+10.
\]
The leading principal minors of the Hurwitz matrix associated with $f$, i.e.%
\[
H_{f}=\left(
\begin{array}
[c]{lll}%
3 & 1 & 0\\
10 & 7 & 3\\
0 & 0 & 10
\end{array}
\right)
\]
are all positive and thus, by the Routh--Hurwitz criterion, $f$ is stable. It
can be also easily checked, that $f\in\mathcal{W}_{3}^{\alpha}$ if and only if
$\alpha>\frac{10}{21}\approx\allowbreak0.476\,19$. It means that
$f\notin\mathcal{W}_{3}^{\alpha^{\ast}}$, where $\alpha^{\ast}\approx0.46557$
is a positive number defined in (\ref{w11a}). Moreover, by
Conclusion~\ref{con1}, we know that there exists a polynomial $g\in
\mathbb{R}_{n-4}^{+}$ such that the polynomial $F\in\mathbb{R}_{n}^{+}$ of the
form $F\left(  s\right)  =f\left(  s\right)  s^{n-3}+g\left(  s\right)  $ is
stable. Since the $2$-by-$2$ leading principal submatrix of the Hurwitz matrix
$H_{F}$ is identical to that of $H_{f}$, we obtain that $\mathcal{W}%
_{n}^{\alpha^{\ast}}\subsetneq\mathcal{H}_{n}$ for $n\geq4$.
\end{example}

\begin{example}
\label{ex3}One can easily check, for example using the Routh--Hurwitz
criterion, that the polynomial $f\in\mathbb{R}_{6}^{+}$ of the form
\[
f\left(  s\right)  =2s^{6}+6s^{5}+12s^{4}+16s^{3}+12s^{2}+10s+1
\]
is stable, whereas the polynomial
\[
f_{0}\left(  s\right)  =6s^{5}+12s^{4}+16s^{3}+12s^{2}+10s+1
\]
is not ($\triangle_{4}\left(  f_{0}\right)  =-516$). This shows that the
stability of the polynomial $f$ does not imply the stability of the
polynomials $f_{0},\ldots,f_{n-m}$ given by (\ref{w22}).
\end{example}

\begin{example}
\label{ex4} Let $f\in\mathbb{R}_{8}^{+}$ be of the form
\begin{align*}
f\left(  s\right)   &  =s^{8}+s^{7}+46s^{6}+34.5s^{5}+791s^{4}+\\
&  +395.75s^{3}+6026s^{2}+1509.375s+17160.
\end{align*}
It was recently shown (and can be easily verified, e.g. by the Routh--Hurwitz
criterion) that the polynomial $f$ is stable whereas the polynomial
$f^{\left[  1.139\right]  }$ is not (see Bia\l as and Bia\l as--Cie\.{z}
\cite{Bia-Cie}). Also, it follows from the stability of $f$ that $\lambda
_{i}\left(  f\right)  <1$ and thus%
\[
\lambda_{i}\left(  f^{\left[  0.139\right]  }\right)  =\lambda_{i}%
^{0.139}\left(  f\right)  <1\text{,}%
\]
for $i=2,\ldots,7$. It means that $f\in\mathcal{H}_{8}$ and $f^{\left[
0.139\right]  }\in\mathcal{W}_{8}$ but $f^{\left[  1.139\right]  }=f\circ
f^{\left[  0.139\right]  }\notin\mathcal{H}_{8}$ proving that
Theorem~\ref{th1} does not hold in the general case.
\end{example}

\begin{example}
\label{ex5} In order to illustrate Theorem \ref{th6}, consider the polynomial
\[
f\left(  s\right)  =2s^{5}+4s^{4}+4s^{3}+4s^{2}+2s+1.
\]
Since $\triangle_{4}\left(  f\right)  <0$, $f$ is not stable. According to
(\ref{w101}) we have%
\[
\lambda_{2}\left(  f\right)  =\lambda_{3}\left(  f\right)  =\lambda_{4}\left(
f\right)  =\frac{1}{2}%
\]
and thus $f\in\mathcal{W}_{5}$. Moreover, by (\ref{ww1}),
\[
p^{\ast}=\frac{\log\alpha^{\ast}}{\log\max_{2\leq i\leq4}\lambda_{i}\left(
f\right)  }\approx1.102\,9.
\]
Taking, for example, $m=4$ we can define the polynomial (for more details,
return to the proof of Theorem \ref{th6})%
\[
g\left(  s\right)  =(f_{0}^{[2]}\circ f_{1}^{[2]})\left(  s\right)
=2^{6}s^{4}+2^{8}s^{3}+2^{8}s^{2}+2^{6}s+2^{2}%
\]
which is stable and such that the generalized Hadamard product $f\bullet g$ is
stable too. Indeed, $f\bullet g=\left\{  F_{0},F_{1}\right\}  $ where%
\begin{align*}
F_{0}\left(  s\right)   &  =2^{8}s^{4}+2^{10}s^{3}+2^{10}s^{2}+2^{7}s+2^{2}\\
F_{1}\left(  s\right)   &  =2^{7}s^{4}+2^{10}s^{3}+2^{10}s^{2}+2^{8}s+2^{3}%
\end{align*}
are stable.
\end{example}

\section*{Acknowledgments}

This research work was partially supported by the Faculty of Applied
Mathematics AGH UST statutory tasks (grant no. 11.11.420.004) within subsidy
of Ministry of Science and Higher Education.

\section*{Appendix}

Let $n$ be a positive integer and let $f\in\mathcal{H}_{n}$. We shall prove
now that there exists a positive number $\varepsilon>0$ such that for every
$a_{n+1}\in\left(  0,\varepsilon\right)  $ the polynomial%
\[
F\left(  s\right)  =f\left(  s\right)  +a_{n+1}s^{n+1}%
\]
is stable. In fact, it is a little bit more than we need in the proof of Lemma
\ref{lem1}. Recall also, that the proof of this fact based on the
Hermite--Biehler theorem can be found in \cite{Bhat} (see Lemma 5.3 therein).

Suppose that $f$ is of the form (\ref{w1}) and let $F_{\alpha}\left(
s\right)  =f\left(  s\right)  +\alpha s^{n+1}$. The Hurwitz matrix associated
with $F_{\alpha}$ has the form%
\[
H_{F_{\alpha}}=%
\begin{pmatrix}
a_{n} & \alpha & 0 & 0 & \ldots & 0\\
a_{n-2} & a_{n-1} & a_{n} & \alpha & \ldots & 0\\
a_{n-4} & a_{n-3} & a_{n-2} & a_{n-1} & \ldots & 0\\
\vdots & \vdots & a_{n-4} & a_{n-3} & \ldots & 0\\
\vdots & \vdots & \vdots & \vdots & \ddots & \vdots\\
0 & 0 & 0 & 0 & \ldots & a_{0}%
\end{pmatrix}
.
\]
All leading principal minors of the matrix $H_{F_{\alpha}}$ are continuous
functions of $\alpha$. Moreover, we have $\triangle_{1}\left(  F_{\alpha
}\right)  =a_{n}$ and for $\alpha=0$ and $k=1,\ldots,n-1$
\[
\triangle_{k+1}\left(  F_{0}\right)  =a_{n}\triangle_{k}\left(  f\right)
\text{.}%
\]
It follows from the stability of $f$, that there exist positive numbers
$\varepsilon_{2},\ldots,\varepsilon_{n}$ such that for $k=2,\ldots,n$
\[
\triangle_{k}\left(  F_{\alpha}\right)  >0\text{,\quad for every }\alpha
\in\left(  0,\varepsilon_{k}\right)  \text{.}%
\]
Putting $\varepsilon=\min\left\{  \varepsilon_{2},\ldots,\varepsilon
_{n},1\right\}  $, we get by the Routh--Hurwitz criterion, that for any
$a_{n+1}\in\left(  0,\varepsilon\right)  $ the polynomial $F\left(  s\right)
=f\left(  s\right)  +a_{n+1}s^{n+1}$ is stable. This completes the proof.

\end{document}